\documentclass[3p,times,a4paper]{my-elsarticle}

\usepackage{a4wide}

\usepackage[colorlinks=true,linkcolor=black,citecolor=blue,urlcolor=blue]{hyperref}

\journal{Mathematics and Computers in Simulation}

\usepackage{amsmath,amssymb,amsthm}

\usepackage{subfig}
\usepackage{graphicx}
\usepackage{epstopdf}

\newtheorem{theorem}{Theorem}[section]

\newtheorem{lemma}[theorem]{Lemma}
\newtheorem{property}[theorem]{Property}

\theoremstyle{remark}

\theoremstyle{definition}
\newtheorem{definition}{Definition}[section]

\begin{document}
	
\begin{frontmatter}
		
\title{Stability of a fractional HIV/AIDS model}

\author[CIDMA]{Cristiana J. Silva}
\ead{cjoaosilva@ua.pt}

\author[CIDMA]{Delfim F. M. Torres\corref{correspondingauthor}}
\cortext[correspondingauthor]{Corresponding author.}
\ead{delfim@ua.pt}

\address[CIDMA]{Center for Research and Development in Mathematics and Applications (CIDMA),\\
Department of Mathematics, University of Aveiro, 3810-193 Aveiro, Portugal}

\begin{abstract}
We propose a fractional order model for HIV/AIDS transmission. 
Local and uniform stability of the fractional order model is studied. 
The theoretical results are illustrated through numerical simulations. 
\end{abstract}


\end{frontmatter}

% ---------------------------------------------
	
\section{Introduction}

Fractional differential equations (FDEs), also known in the literature
as extraordinary differential equations, are a generalization of differential 
equations through the application of fractional calculus, that is, the branch 
of mathematical analysis that studies different possibilities of defining 
differentiation operators of noninteger order \cite{MR3443073,MR1356813}. 
FDEs are naturally related to systems with memory, which explains 
their usefulness in most biological systems \cite{MR3722341}.
Indeed, FDEs have been considered in many epidemiological models. 
In \cite{Ahmed2007}, a fractional order model for nonlocal epidemics is considered, 
and the results are expected to be relevant to foot-and-mouth disease, 
SARS, and avian flu. Some necessary and sufficient conditions for local stability 
of fractional order differential systems are provided \cite{Ahmed2007}. In \cite{Ozalp:2011}, 
a fractional order SEIR model with vertical transmission within a nonconstant 
population is considered, and the asymptotic stability of the disease free 
and endemic equilibria are analyzed. The stability of the endemic 
equilibrium of a fractional order SIR model is studied in \cite{Guo:fracSIR:2017}. 
A fractional order model of HIV infection of CD4+ T-cells is analyzed in \cite{Ding:Ye:2009}. 
A fractional order predator prey model and a fractional order rabies model are proposed in
\cite{Ahmed:etall:predatorprey:2207}. The stability of equilibrium points are
studied, and an example is given where the equilibrium point is a centre 
for the integer order system but locally asymptotically stable 
for its fractional-order counterpart \cite{Ahmed:etall:predatorprey:2207}.
A fractional control model for malaria transmission is proposed 
and studied numerically in \cite{PintoMachado:Malaria:2013}. 

The question of stability for FDEs is crucial:
see,  e.g., \cite{LiZhang:2011,Rivero:etall:2013} 
for good overviews on stability of linear/nonlinear, positive, with delay, 
distributed, and continuous/discrete fractional order systems. 
In \cite{Delavari2012}, an extension of the Lyapunov direct method 
for fractional-order systems using Bihari's and Bellman--Gronwall's 
inequality, and a proof of a comparison theorem for fractional-order
systems, are obtained. A new lemma for Caputo fractional derivatives, 
when $0 < \alpha < 1$, is proposed in \cite{Aguila-Camacho:2014}, 
which allows to find Lyapunov candidate functions for proving 
the stability of many fractional order systems, using the 
fractional-order extension of the Lyapunov direct method.
Motivated by the work \cite{Aguila-Camacho:2014}, the authors of 
\cite{VargasDeLeon201575} extended the Volterra-type Lyapunov function 
to fractional-order biological systems through an inequality to estimate 
the Caputo fractional derivatives of order $\alpha \in (0, 1)$. 
Using this result, the uniform asymptotic stability of some Caputo-type 
epidemic systems with a pair of fractional-order differential equations is proved. 
Such systems are the basic models of infectious disease dynamics 
(SIS, SIR and SIRS models) and Ross--Macdonald model for vector-borne diseases.
For more on the subject see \cite{MR3173287}, where
the problem of output feedback stabilization for fractional order linear 
time-invariant systems with fractional commensurate order is investigated,
and \cite{MR3186983}, where the stability of a special observer with 
a nonlinear weighted function and a transient dynamics function is 
rigorously analyzed for slowly varying disturbances and higher-order 
disturbances of fractional-order systems.

Here we propose a Caputo fractional order SICA epidemiological model 
with constant recruitment rate, mass action incidence and variable population size, 
for HIV/AIDS transmission. The model is based on an integer-order 
HIV/AIDS model without memory effects firstly proposed in \cite{SilvaTorres:TBHIV:2015} 
and later modified in \cite{SILVA201770,SilvaTorres:HIV:Ankara:2018}.
The model for $\alpha = 1$ describes well the clinical reality given by the data 
of HIV/AIDS infection in Cape Verde from 1987 to 2014 \cite{SILVA201770}.
In the present work, we extend the model by considering fractional differentiation, 
in order to capture memory effects, long-rage interactions, and hereditary properties,
which exist in the process of HIV/AIDS transmission but are neglected in the case
$\alpha = 1$, that is, for integer-order differentiation \cite{MR3872489,MR3780644}. 
Using the results from \cite{Matignon:1996} and \cite{Ahmed2007}, we prove the local asymptotic 
stability of the disease free equilibrium. Then, we extend the results 
of \cite{Delavari2012} and \cite{VargasDeLeon201575} and prove the uniform 
asymptotic stability of the disease free and endemic equilibrium points. 
For the numerical implementation of the fractional derivatives, 
we have used the Adams--Bashforth--Moulton scheme, which has been 
implemented in the \textsf{fde12} \textsf{Matlab} routine by Garrappa \cite{Garrappa}.
The software code implements a predictor-corrector PECE method, 
as described in \cite{Diethelm:1999}.

The paper is organized as follows. In Section~\ref{sec:pre}, we present basic 
definitions and recall necessary results on Caputo fractional calculus and 
local and uniform asymptotic stability and Volterra-type 
Lyapunov functions for fractional-order systems. The original
results appear in Section~\ref{sec:frac:model}:
we introduce our Caputo fractional-order HIV/AIDS model  
and study the existence of equilibrium points. 
More precisely, in Section~\ref{sec:local:DFE} 
we prove local asymptotic stability of the disease 
free equilibrium, while in Sections~\ref{key} and \ref{sec:unif:stab} 
we prove uniform asymptotic stability of the disease free 
and endemic equilibrium points, respectively. 
We end with Section~\ref{sec:numsim} of numerical simulations, 
which illustrate the stability results proved 
in Sections~\ref{sec:local:DFE}--\ref{sec:unif:stab}. 

% -----------------------------------------------------

\section{Preliminaries on the Caputo fractional calculus}
\label{sec:pre}

We begin by introducing the definition of Caputo fractional derivative
and recalling its main properties.

\begin{definition}[See \cite{GJI:GJI529}]
Let $a > 0$, $t > a$, and $\alpha, a, t \in\mathbb{R} $. The Caputo fractional 
derivative of order $\alpha$ of a function $f \in C^n$ is given by
\begin{equation*}
_{a}^{C}D_{t}^{\alpha}f(t)= \dfrac{1}{\Gamma(n-\alpha)} 
\int_{a}^{t}\dfrac{f^{(n)}(\xi)}{(t-\xi)^{\alpha+1-n}}d\xi,
\qquad n-1<\alpha<n \in \mathbb{N}.
\end{equation*}
\end{definition}

\begin{property}[Linearity; see, e.g., \cite{Diethelm}] 
\label{propertyLinear}
Let $f,g : [a,b] \rightarrow \mathbb{R}$ be such that 
$_{a}^{C}D_{t}^{\alpha}f(t)$ and $_{a}^{C}D_{t}^{\alpha}g(t)$ 
exist almost everywhere and let $c_1,c_2 \in \mathbb{R}$. Then, 
$_{a}^{C}D_{t}^{\alpha}(c_1 f(t)+ c_2 g(t))$ exists almost 
everywhere with
\begin{equation*}
_{a}^{C}D_{t}^{\alpha} (c_{1} f(t)+ c_{2} g(t))
= c_{1} \, _{a}^{C}D_{t}^{\alpha}f(t) 
+ c_{2}\, _{a}^{C}D_{t}^{\alpha}g(t).
\end{equation*}
\end{property}

\begin{property}[Caputo derivative of a constant; see, e.g., \cite{Podlubny}] 
The fractional derivative of a constant function $f(t)\equiv c$ is zero:
\begin{equation*}
_{a}^{C}D_{t}^{\alpha} c = 0.
\end{equation*}
\end{property}

Let us consider the following general fractional differential 
equation involving the Caputo derivative:
\begin{equation}
\label{CaputoGeneral}
_{a}^{C}D_{t}^{\alpha} x(t)= f(t,x(t)), 
\qquad \alpha \in (0,1),
\end{equation}
subject to a given initial condition $x_0=x(t_0)$.

\begin{definition}[See, e.g., \cite{LiChen:Automatica:2009}]
\label{def:eq}
The constant $x^*$ is an equilibrium point of the Caputo fractional 
dynamic system \eqref{CaputoGeneral} if, and only if, $f(t, x^*) = 0$. 
\end{definition}

Following \cite{Matignon:1996}, an equilibrium point $x^*$ of the Caputo fractional 
dynamic system \eqref{CaputoGeneral} is locally asymptotically stable if all 
the eigenvalues $\lambda$ of the Jacobian matrix of system \eqref{CaputoGeneral}, 
evaluated at the equilibrium point $x^*$, satisfies the following condition:
\begin{equation}
\label{cond:alphapi2}
|\arg (\lambda)| > \frac{\alpha \pi}{2}.
\end{equation}

Next theorem gives an extension of the celebrated Lyapunov direct method
for Caputo type fractional order nonlinear systems \cite{Delavari2012}. 

\begin{theorem}[Uniform Asymptotic Stability \cite{Delavari2012}]
\label{uniform_stability}
Let $x^*$ be an equilibrium point for the nonautonomous fractional 
order system \eqref{CaputoGeneral} and $\Omega \subset \mathbb{R}^{n}$ 
be a domain containing $x^*$. Let $L:[0,\infty) \times \Omega 
\rightarrow \mathbb{R}$ be a continuously differentiable function
such that
\begin{equation*}
W_1(x) \leq L(t,x(t)) \leq W_2(x)
\end{equation*}
and
\begin{equation*}
_{a}^{C}D_{t}^{\alpha}  L(t,x(t)) \leq -W_3(x)
\end{equation*}
for all $\alpha\in (0,1)$ and all $x \in \Omega$, 
where $W_1(\cdot)$, $W_2(\cdot)$ and $W_3(\cdot)$ 
are continuous positive definite functions on $\Omega$. 
Then the equilibrium point $x^*$ of system \eqref{CaputoGeneral} 
is uniformly asymptotically stable.
\end{theorem} 

In what follows, we recall a lemma proved in \cite{VargasDeLeon201575}, 
where a Volterra-type Lyapunov function is obtained 
for fractional-order epidemic systems. 

\begin{lemma}[See \cite{VargasDeLeon201575}]
\label{lemmaVargas}
Let $x(\cdot)$ be a continuous and differentiable function
with $x(t)\in \mathbb{R_{+}}$. Then, for any time instant 
$t \geq t_0$, one has
\begin{equation*}
_{t_{0}}^{C}D_{t}^{\alpha}\left[ x(t)-x^{*}-x^{*}
\ln\dfrac{x(t)}{x^{*}}\right]
\leq \left( 1-\dfrac{x^{*}}{x(t)}\right) \,
{_{t_{0}}^{C}D}_{t}^{\alpha} x(t), 
\qquad x^{*} \in \mathbb{R}^{+}, 
\qquad \forall\alpha\in (0,1).
\end{equation*}
\end{lemma}

% ---------------------------------------

\section{The fractional HIV/AIDS model}
\label{sec:frac:model}

In this section we propose a Caputo fractional-order model for HIV/AIDS
with memory effects. Our population model assumes a constant recruitment rate, 
mass action incidence, and variable population size. 

The model subdivides human population into four mutually-exclusive 
compartments: susceptible individuals ($S$); 
HIV-infected individuals with no clinical symptoms of AIDS 
(the virus is living or developing in the individuals 
but without producing symptoms or only mild ones) 
but able to transmit HIV to others ($I$); 
HIV-infected individuals under ART treatment (the so called 
chronic stage) with a viral load remaining low ($C$); 
and HIV-infected individuals with AIDS clinical symptoms ($A$).
The total population at time $t$, denoted by $N(t)$, is given by
$N(t) = S(t) + I(t) + C(t) + A(t)$.
Effective contact with people infected with HIV is at a rate $\lambda$, given by
\begin{equation*}
\lambda = \beta \left( I + \eta_C \, C  + \eta_A  A \right),
\end{equation*}
where $\beta$ is the effective contact rate for HIV transmission.
The modification parameter $\eta_A \geq 1$ accounts for the relative
infectiousness of individuals with AIDS symptoms, in comparison to those
infected with HIV with no AIDS symptoms. Individuals with AIDS symptoms
are more infectious than HIV-infected individuals (pre-AIDS) because
they have a higher viral load and there is a positive correlation
between viral load and infectiousness. 
On the other hand, $\eta_C \leq 1$ translates the partial restoration 
of immune function of individuals with HIV infection that use ART correctly.
All individuals suffer from natural death, at a constant rate $\mu$. 
We assume that HIV-infected individuals 
with and without AIDS symptoms have access to ART treatment. 
HIV-infected individuals with no AIDS symptoms $I$ progress to the class 
of individuals with HIV infection under ART treatment $C$ at a rate $\phi$, 
and HIV-infected individuals with AIDS symptoms are treated for HIV at rate $\gamma$.
Individuals in the class $C$ leave to the class $I$ at a rate $\omega$.
We also assume that an HIV-infected individual with AIDS symptoms $A$ 
that starts treatment moves to the class of HIV-infected individuals $I$, 
moving to the chronic class $C$ only if the treatment is maintained. 
HIV-infected individuals with no AIDS symptoms $I$ that do not take 
ART treatment progress to the AIDS class $A$ at rate $\rho$. Note that 
only HIV-infected individuals with AIDS symptoms $A$ 
suffer from an AIDS induced death, at a rate $d$.
The Caputo fractional-order system that describes 
the previous assumptions is: 
\begin{equation}
\label{mod:frac}
\begin{cases}
_{t_{0}}^{C}D_{t}^{\alpha}S(t) 
= \Lambda - \beta \left( I(t) + \eta_C \, C(t)  
+ \eta_A  A(t) \right) S(t) - \mu S(t),\\[0.2 cm]
_{t_{0}}^{C}D_{t}^{\alpha}I(t) 
= \beta \left( I(t) + \eta_C \, C(t)  + \eta_A  A(t) \right) S(t) 
- \left(\rho + \phi + \mu\right) I(t) + \omega C(t) + \gamma A(t), \\[0.2 cm]
_{t_{0}}^{C}D_{t}^{\alpha}C(t) = \phi I(t) - (\omega + \mu)C(t),\\[0.2 cm]
_{t_{0}}^{C}D_{t}^{\alpha}A(t) = \rho \, I(t) - (\gamma + \mu + d) A(t).
\end{cases}
\end{equation}
The biologically feasible region of system \eqref{mod:frac} is given by 
\begin{equation}
\label{Omega:inv:region:HIV}
\Omega = \left\{ \left( S, I, C, A \right) \in \mathbb{R}_+^{4} 
\, : \, N \leq \frac{\Lambda}{\mu} \right\}.
\end{equation}
The model \eqref{mod:frac} has a disease free equilibrium given by
\begin{equation}
\label{DFE}
\Sigma_0 = \left(S^0, I^0, C^0, A^0 \right)
= \left(\frac{\Lambda}{\mu},0, 0,0  \right).
\end{equation}
Let
\begin{equation}
\label{eq:R0:model:1}
R_0 = \frac{ S^0 \beta\, \left(  \xi_2  \left( \xi_1 +\rho\, \eta_A \right) 
+ \eta_C \,\phi \, \xi_1 \right) }{\mu\, \left(  \xi_2  \left( \rho + \xi_1
\right) +\phi\, \xi_1 +\rho\,d \right) +\rho\,\omega\,d} 
= \frac{S^0 \mathcal{N}}{\mathcal{D}},
\end{equation}
where 
$\xi_1 = \gamma + \mu + d$, $\xi_2 = \omega + \mu$, 
$$
\mathcal{N} = \beta\, \left(  \xi_2  \left( \xi_1 +\rho\, \eta_A \right) 
+ \eta_C \,\phi \, \xi_1 \right)
$$ 
and 
$$
\mathcal{D} 
= \mu\, \left(  \xi_2  \left( \rho + \xi_1\right) 
+\phi\, \xi_1 +\rho\,d \right) +\rho\, \omega\,d.
$$
Whenever $R_0 > 1$, the model \eqref{mod:frac} has a unique endemic equilibrium 
$\Sigma_* = \left(S^*, I^*, C^*, A^* \right)$ given by
\begin{equation}
\label{EE}
S^* = \frac{ \mathcal{D}}{ \mathcal{N}}\, , \quad 
I^* = \frac{\xi_1 \xi_2 (\Lambda \mathcal{N} 
- \mu \mathcal{D})}{\mathcal{D} \mathcal{N}} \, , \quad
C^* = \frac{\phi \xi_1 (\Lambda \mathcal{N} 
-\mu \mathcal{D})}{\mathcal{D} \mathcal{N}} \, , \quad
A^* = \frac{\rho \xi_2 \left(
\Lambda \mathcal{N} - \mu \mathcal{D}\right)}{\mathcal{D} \mathcal{N}}.
\end{equation}	

% ------------------------

\subsection{Local asymptotic stability of the disease free equilibrium $\Sigma_0$}
\label{sec:local:DFE}

As firstly proved in \cite{Matignon:1996},
stability is guaranteed if and only if the roots
of some polynomial (the eigenvalues of the matrix of dynamics or the poles
of the transfer matrix) lie \emph{outside} the closed angular sector 
$|\arg (\lambda)| \leq \frac{\alpha \pi}{2}$. In our case,
the Jacobian matrix $J (\Sigma_0)$ for system \eqref{mod:frac} evaluated 
at the uninfected steady state $\Sigma_0$ \eqref{DFE} is given by
\begin{equation}
\label{eq:JacDFE}
J (\Sigma_0) = 
\left[ 
\begin{array}{cccc} 
-\mu&-{\frac {\beta\,\Lambda}{\mu}}
&-{\frac {\beta\,\Lambda\, \eta_C }{\mu}}
&-{\frac {\Lambda\,\beta\,\eta_A }{\mu}}\\ 
\noalign{\medskip}
0&{\frac {\Lambda\,\beta}{\mu}}-\mu-\phi-\rho
&{\frac {\Lambda\,\beta\, \eta_C }{\mu}}+\omega
&{\frac{\Lambda\,\beta\, \eta_A }{\mu}}+\alpha\\ 
\noalign{\medskip}0&\phi&-\omega-\mu&0\\ 
\noalign{\medskip}
0&\rho&0&-\gamma-\mu-d
\end {array}\right].
\end{equation}
The uninfected steady state is asymptotically stable if all of the eigenvalues 
$\lambda$ of the Jacobian matrix $J (\Sigma_0)$ satisfy the following 
condition (see, e.g., \cite{Matignon:1996}): 
\begin{equation*}
|\arg (\lambda)| > \frac{\alpha \pi}{2}.
\end{equation*}
Let $\xi_3 = \rho + \phi + \mu$. The eigenvalues are determined 
by solving the characteristic equation $\det ( J ( \Sigma_0 ) − \lambda I ) = 0$. 
For $J( \Sigma_0 )$ as in \eqref{eq:JacDFE}, the characteristic equation is given by
$$
q \, p = 0
$$
with 
\begin{equation}
\label{eq:poly:q}
q = (\lambda+\mu)
\end{equation}
and
\begin{equation}
\label{eq:poly:p}
p = \lambda^3 + b_1 \lambda^2+ b_2 \lambda+ b_3,
\end{equation}
where
\begin{equation*}
\begin{split}
b_1 &= -\frac{\Lambda\,\beta  - \mu ( \xi_1 + \xi_2 + \xi_3)}{\mu} \, ,\\
b_2 &= -\frac{1}{\mu}\left( \Lambda \beta \left(  \eta_A \,\rho+ \eta_C \,
\phi+ \xi_1 + \xi_2 \right) -\mu\, \left( d ( \xi_2 + \xi_3)
+ \gamma (\mu + \xi_2 +\phi) + \mu (\xi_2 + \omega + 2 \xi_3) 
+\omega\,\rho \right) \right), \\
b_3 &= -\frac{1}{\mu}\left( \Lambda \mathcal{N} - \mu \mathcal{D} \right).
\end{split}
\end{equation*}
From \eqref{eq:poly:q} we have that the eigenvalue $\lambda_1 = -\mu$ satisfies 
$|\arg (\lambda_1)| > \frac{\alpha \pi}{2}$ for all $\alpha \in (0, 1)$. 
The discriminant $D(p)$ of the polynomial \eqref{eq:poly:p} 
is given (see \cite{Ahmed2007}) by 
\begin{equation*} 
D(p) = - \left| \begin{array}{c c c c c}
1 & b_1 &  b_2 &  b_3 &  0\\
0 & 1  & b_1 &  b_2 &  b_3\\
3 & 2 b_1 & b_2 &  0 & 0\\
0 & 3 & 2 b_1 &  b_2 & 0\\
0 & 0 & 3 & 2 b_1 & b_2
\end{array} \right| =
18 b_1 b_2 b_3+(b_1 b_2)^2-4 b_3 b_1^3 - 4 b_2^3 - 27 b_3^3.
\end{equation*}
Following \cite{Ahmed2007}, all roots of the polynomial \eqref{eq:poly:p} satisfy 
condition \eqref{cond:alphapi2} if the following conditions hold: 
\begin{itemize}
\item[(i)] if $D(p) > 0$, then the Routh--Hurwitz conditions are 
a necessary and sufficient condition for the equilibrium point $\Sigma_0$  
to be locally asymptotically stable, i.e., $b_1 > 0$, $b_3 > 0$ and $b_1 b_2 − b_3 > 0$;
	
\item[(ii)] if $D(p) < 0$, $b_1 \geq 0$, $b_2 \geq 0$, $b_3 > 0$, and $\alpha < 2/3$, 
then $\Sigma_0$ is locally asymptotically stable; 
	
\item[(iii)] if $D(p) < 0$, $b_1 < 0$, $b_2 < 0$, and $\alpha > 2/3$, 
then $\Sigma_0$ is unstable; 

\item[(iv)] if $D(p) < 0$, $b_1 > 0$, $b_2 > 0$, and $b_1 b_2 − b_3 = 0$, 
then $\Sigma_0$ is locally asymptotically stable for all $\alpha \in[0, 1)$;

\item[(v)] $b_3 >0$ is a necessary condition for local asymptotic stability of $\Sigma_0$. 
\end{itemize}

% ---------------------------------- 

\subsection{Uniform asymptotic stability of the disease free equilibrium $\Sigma_0$}
\label{key}

In this section, we prove the uniform asymptotic stability of the disease 
free equilibrium $\Sigma_0$ \eqref{DFE} of the fractional order system \eqref{mod:frac}. 

\begin{theorem}
\label{theor:DFE}
Let $\alpha \in (0, 1)$. The disease free equilibrium $\Sigma_0$ \eqref{DFE}, 
of the fractional system \eqref{mod:frac}, is uniformly asymptotically 
stable in $\Omega$ \eqref{Omega:inv:region:HIV}, whenever \eqref{eq:R0:model:1} 
satisfies $R_0 < 1$.
\end{theorem}

\begin{proof}
Consider the following Lyapunov function:
\begin{equation*}
V(t) =  c_1 I(t) + c_2 C(t) + c_3 A(t),
\end{equation*}
where 
\begin{equation*}
\begin{split}
c_1 &= \xi_1 \xi_2 + \xi_1 \phi \eta_C + \xi_2 \rho \eta_A, \\
c_2 &= \xi_1 \omega + \xi_1 \xi_3 \eta_C + \rho \eta_A \omega - \eta_C \rho \gamma, \\
c_3 &= \gamma \xi_2 + \xi_2 \xi_3 \eta_A + \phi \eta_C \gamma - \phi \eta_A \omega.
\end{split}
\end{equation*}
Function $V$ is defined, continuous and positive 
definite for all $I(t) > 0$, $C(t) > 0$ and $A(t) > 0$. 
By Property~\ref{propertyLinear}, we have 
\begin{equation*}
_{t_{0}}^{C}D_{t}^{\alpha}V  
=  c_1  {_{a}^{C}D}_{t}^{\gamma} I 
+ c_2 \, _{a}^{C}D_{t}^{\gamma}C  
+ c_3 \, _{a}^{C}D_{t}^{\gamma}A.
\end{equation*}
From \eqref{mod:frac} we have
\begin{equation*}
_{t_{0}}^{C}D_{t}^{\alpha}V  =  c_1 \left( \beta \left( I + \eta_C \, C  
+ \eta_A  A \right) S - \xi_3 I + \gamma A + \omega C \right) 
+ c_2 \left( \phi I - \xi_2 C \right)  + c_3 \left( \rho \, I - \xi_1 A \right). 
\end{equation*}
Note that 
$$
\xi_1 \omega + \xi_1 \xi_3 \eta_C + \rho \eta_A \omega 
- \eta_C \rho \gamma = \xi_1 \omega + \gamma (\phi + \mu)\eta_C 
+ (\mu + d) \xi_3 \eta_C + \rho \eta_A \omega > 0
$$ 
and 
$$
\gamma \xi_2 + \xi_2 \xi_3 \eta_A + \phi \eta_C \gamma - \phi \eta_A \omega 
= \gamma \xi_2 + \omega (\rho + \mu)\eta_A 
+ \mu \xi_3 \eta_A + \phi \eta_C \gamma > 0. 
$$
Therefore, we have 
\begin{equation*}
\begin{split}
_{t_{0}}^{C}D_{t}^{\alpha}V  
&= (\xi_1 \xi_2 \beta +  \xi_1 \phi \eta_C \beta  + \xi_2 \rho \eta_A \beta) I S 
+ (- \xi_1 \xi_2 \xi_3 + \xi_1 \omega \phi + \gamma \xi_2 \rho ) I\\
& \quad + \eta_C (\xi_1 \xi_2 \beta + \xi_1 \phi \eta_C \beta 
+ \xi_2 \rho \eta_A \beta ) C S + \eta_C (- \xi_1 \xi_3 \xi_2 
+ \xi_1 \phi \omega +  \rho \gamma \xi_2) C\\
& \quad + \eta_A (\xi_1 \xi_2 \beta + \xi_1 \phi \eta_C \beta 
+ \xi_2 \rho \eta_A \beta ) A S + \eta_A (- \xi_2 \xi_3 \xi_1 
+ \phi \omega \xi_1 + \xi_2 \rho \gamma ) A.
\end{split}
\end{equation*}
As $S \leq S^0$, 
\begin{equation*}
\begin{split}
_{t_{0}}^{C}D_{t}^{\alpha}V  
&\leq  (\xi_1 \xi_2 \beta +  \xi_1 \phi \eta_C \beta  
+ \xi_2 \rho \eta_A \beta) I S^0 + \left(- \xi_1 \xi_2 \xi_3 
+ \xi_1 \omega \phi + \gamma \xi_2 \rho \right) I\\
& \quad + \eta_C (\xi_1 \xi_2 \beta + \xi_1 \phi \eta_C \beta 
+ \xi_2 \rho \eta_A \beta ) C S^0 + \eta_C \left(
- \xi_1 \xi_3 \xi_2 + \xi_1 \phi \omega +  \rho \gamma \xi_2\right) C\\
& \quad + \eta_A \left(\xi_1 \xi_2 \beta + \xi_1 \phi \eta_C \beta 
+ \xi_2 \rho \eta_A \beta \right) A S^0 + \eta_A 
\left(- \xi_2 \xi_3 \xi_1 + \phi \omega \xi_1 + \xi_2 \rho \gamma \right) A 
\end{split}
\end{equation*}
holds. From  $S^0 \left(\xi_1 \xi_2 \beta +  \xi_1 \phi \eta_C \beta  
+ \xi_2 \rho \eta_A \beta\right) = \mathcal{N}$ 
and $- \xi_1 \xi_2 \xi_3 + \xi_1 \omega \phi 
+ \gamma \xi_2 \rho = - \mathcal{D}$, one has
\begin{equation*}
\begin{split}
_{t_{0}}^{C}D_{t}^{\alpha}V 
&\leq \mathcal{N} I - \mathcal{D} I  
+ \eta_C \left(\mathcal{N} C - \mathcal{D} C \right)
+ \eta_A \left(\mathcal{N} A - \mathcal{D} A \right)\\
&= \mathcal{D} I \left( R_0 - 1 \right)  
+ \eta_C \mathcal{D} C \left( R_0 - 1 \right) 
+ \eta_A \mathcal{D} A \left( R_0 - 1 \right)\\
&\leq 0 \, \, \text{for} \, \, R_0 < 1.
\end{split}
\end{equation*}
Because all the model parameters are nonnegative, it follows that 
$_{t_{0}}^{C}D_{t}^{\alpha}V \leq 0$ for $R_0 < 1$ with 
$_{t_{0}}^{C}D_{t}^{\alpha}V=0$ if, and only if, $I =C=A=0$. 
Substituting $(I, C, A) = (0, 0, 0)$ in \eqref{mod:frac} shows that 
$S \to S^0=\frac{\Lambda}{\mu}$ as $t \to \infty$. Hence, 
by Theorem~\ref{uniform_stability}, the equilibrium point $\Sigma_0$ 
of system \eqref{mod:frac} is uniformly asymptotically 
stable in $\Omega$, whenever $R_0 < 1$. 	
\end{proof}

% ------------------------

\subsection{Uniform asymptotic stability of the endemic equilibrium $\Sigma_*$}
\label{sec:unif:stab}

In this section we prove uniform asymptotic stability of the endemic equilibrium  
$\Sigma_*$ \eqref{EE} of the fractional order system 
\eqref{mod:frac}.

\begin{theorem}
\label{theor:EndEqui}
Let $\alpha \in (0,1)$ and \eqref{eq:R0:model:1} be such that $R_0 > 1$.
Then the unique endemic equilibrium $\Sigma_*$ \eqref{EE}
of the fractional order system \eqref{mod:frac} 
is uniformly asymptotically stable in the interior 
of $\Omega$ \eqref{Omega:inv:region:HIV}.
\end{theorem}

\begin{proof}
Consider the following function:
$$
V(t) =  V_1(S(t)) + V_2(L(t)) + \frac{\omega}{\xi_2} V_3(I(t) ) + \frac{\gamma}{\xi_1} V_4(T(t)),
$$
where
\begin{equation*}
\begin{split}
V_{1}(S(t))&= S - S^* - S^* \ln\left(\frac{S}{S^*} \right),\\
V_{2}(L(t))&= I - I^* - I^* \ln\left(\frac{I}{I^*} \right),\\
V_{3}(I(t))&= C - C^* - C^* \ln\left(\frac{C}{C^*} \right),\\  
V_{4}(T(t))&= A - A^* - A^* \ln\left(\frac{A}{A^*} \right) \, .
\end{split}
\end{equation*}
Function $V$ is a Lyapunov function because it is 
defined, continuous, and positive definite 
for all $S(t)>0$, $I(t)>0$, $C(t)>0$ and $A(t) > 0$.
By Lemma~\ref{lemmaVargas}, we have
\begin{equation*}
_{t_{0}}^{C}D_{t}^{\alpha}V  \leq \left(1-\frac{S^*}{S}\right) \,
{_{t_{0}}^{C}D}_{t}^{\alpha} S + \left(1-\frac{I^*}{I}\right)\, 
{_{t_{0}}^{C}D}_{t}^{\alpha} I 
+ \frac{\omega}{\xi_2} \left(1-\frac{C^*}{C}\right) \,
{_{t_{0}}^{C}D}_{t}^{\alpha} C+ \frac{\gamma}{\xi_1}  
\left(1-\frac{A^*}{A}\right) \,  {_{t_{0}}^{C}D}_{t}^{\alpha} A.
\end{equation*}
It follows from 
\eqref{mod:frac} that
\begin{multline}
\label{eq:difV:1}
_{t_{0}}^{C}D_{t}^{\alpha} V \leq \left(1-\frac{S^*}{S}\right)\left[ 
\Lambda - \beta \left( I + \eta_C \, C  + \eta_A  A \right) S - \mu S \right]\\
+ \left(1-\frac{I^*}{I}\right)\left[  \beta \left( I + \eta_C \, C  
+ \eta_A  A \right) S - \xi_3 I + \gamma A + \omega C \right]\\
+ \frac{\omega}{\xi_2}  \left(1-\frac{C^*}{C}\right)\left[  \phi I -  \xi_2 C \right] 
+ \frac{\gamma}{\xi_1} \left(1-\frac{A^*}{A}\right)\left[  \rho I - \xi_1 A \right].
\end{multline}
Using the relation $\Lambda = \beta \left( I^* + \eta_C \, C^*  
+ \eta_A  A^* \right) S^* + \mu S^*$, we have from the first equation 
of system \eqref{mod:frac} at steady-state that \eqref{eq:difV:1} can be written as
\begin{multline*}
_{t_{0}}^{C}D_{t}^{\alpha} V \leq \left(1-\frac{S^*}{S}\right)\left[ 
\beta \left( I^* + \eta_C \, C^*  + \eta_A  A^* \right) S^* + \mu S^* 
- \beta \left( I + \eta_C \, C  + \eta_A  A \right) S - \mu S \right]\\
+ \left(1-\frac{I^*}{I}\right)\left[  \beta \left( I + \eta_C \, C  
+ \eta_A  A \right) S - \xi_3 I + \gamma A + \omega C \right]\\
+ \frac{\omega}{\xi_2}  \left(1-\frac{C^*}{C}\right)\left[  \phi I -  \xi_2 C \right] 
+ \frac{\gamma}{\xi_1} \left(1-\frac{A^*}{A}\right)\left[  
\rho I - \xi_1 A \right],
\end{multline*}
which can then be simplified to
\begin{multline*}
_{t_{0}}^{C}D_{t}^{\alpha} V  \leq \left(1-\frac{S^*}{S}\right) \beta I^* S^* 
+ \mu S^* \left( 2 - \frac{S}{S^*} - \frac{S^*}{S}\right) - \beta I S + \beta I S^*\\
+ \beta( \eta_C C^* + \eta_A A^*) S^* - \beta (\eta_C C + \eta_A A) S 
- \frac{S^*}{S} \beta (\eta_C C^* + \eta_A A^*) S^* + S^* \beta (\eta_C C + \eta_A A)\\
+ \left(1-\frac{I^*}{I}\right)\left[  \beta \left( I + \eta_C \, C  
+ \eta_A  A \right) S - \xi_3 I + \gamma A + \omega C \right]\\
+ \frac{\omega}{\xi_2}  \left(1-\frac{C^*}{C}\right)\left[  \phi I -  \xi_2 C \right] 
+ \frac{\gamma}{\xi_1} \left(1-\frac{A^*}{A}\right)\left[  \rho I - \xi_1 A \right].
\end{multline*}
Using the relations at the steady state,
\begin{equation*}
\xi_3 I^* = \beta (I^* + \eta_C C^* + \eta_A A^*) S^*  + \gamma A^*  + \omega C^*, 
\quad \xi_2 C^* = \phi I^*, 
\quad \xi_1 A^* = \rho I^*,
\end{equation*} 
and, after some simplifications, we have
\begin{multline*}
_{t_{0}}^{C}D_{t}^{\alpha} V \leq \left( \beta I^* S^* 
+ \mu S^* \right) \left(2 - \frac{S}{S^*} -\frac{S^*}{S} \right) 
+ \beta S^*\left( \eta_C C^* + \eta_A A^* \right) 
\left( 2 - \frac{S^*}{S} - \frac{I}{I^*} \right)\\
+ \beta S^* \left( \eta_C C + \eta_A A \right) 
\left( 1 - \frac{I^*}{I} \frac{S}{S^*}\right)     
+ \gamma A^* \left( 1 - \frac{A}{A^*} \frac{I^*}{I} \right) 
+ \omega C^* \left( 1 - \frac{C}{C^*} \frac{I^*}{I} \right)\\
+  \frac{\omega \phi}{\xi_2} I^* \left( 1 - \frac{I}{I^*} \frac{C^*}{C} \right)
+ \frac{\gamma \rho}{\xi_1} I^* \left( 1 - \frac{I}{I^*} \frac{A^*}{A} \right) .
\end{multline*}
The terms between the larger brackets are less than or equal to zero by the 
well-known inequality that asserts the geometric mean to be less than 
or equal to the arithmetic mean. Therefore, $_{t_{0}}^{C}D_{t}^{\alpha} V(S, I, C, A)$ 
is negative definite when $0 < \alpha < 1$. By Theorem~\ref{uniform_stability}
(the uniform asymptotic stability theorem), the endemic equilibrium $\Sigma_*$ 
is uniformly asymptotically stable in the interior of $\Omega$, whenever $R_0 > 1$. 
\end{proof}

Note that the fractional model \eqref{mod:frac}
is stable independently of the parameter values. Indeed, the 
values of the parameters determine the value of $R_0$ and, 
for $R_0 < 1$, the stability of the system is, according 
with Theorem~\ref{theor:DFE}, ``around'' the disease free 
equilibrium $\Sigma_0$; for $R_0 >1$, the stability of the 
system is, in agreement with Theorem~\ref{theor:EndEqui}, 
``around'' the endemic equilibrium $\Sigma_*$.

% -------------------------------------

\section{Numerical simulations}
\label{sec:numsim}

In this section we study the dynamical behavior of our model \eqref{mod:frac}, 
by variation of the noninteger order derivative $\alpha$.

% ------------------------
\begin{table}[!htb]
\centering
\caption{Parameters values for the HIV/AIDS fractional model \eqref{mod:frac}. 
The parameter $\Lambda$ was estimated and the remaining ones
were taken from \cite{SilvaTorres:PrEP:2018}}.
\label{table:parameters}
\begin{tabular}{l  p{6.5cm} l }
\hline \hline
{\small{Symbol}} &  {\small{Description}} & {\small{Value}} \\
\hline
{\small{$\Lambda$}} & {\small{Recruitment rate}} & {\small{$2.1$}} \\
{\small{$\mu$}} & {\small{Natural death rate}} & {\small{$1/69.54$}} \\
{\small{$\eta_C$}} & {\small{Modification parameter}} & {\small{$0.015$}} \\
{\small{$\eta_A$}} & {\small{Modification parameter}} & {\small{$1.3$}} \\	
{\small{$\phi$}} & {\small{HIV treatment rate for $I$ individuals}} &  {\small{$1$}} \\
{\small{$\rho$}} & {\small{Default treatment rate for $I$ individuals}} & {\small{$0.1 $}} \\
{\small{$\gamma$}} & {\small{AIDS treatment rate}} & {\small{$0.33 $}} \\
{\small{$\omega$}} & {\small{Default treatment rate for $C$ individuals}} & {\small{$0.09$}} \\
{\small{$d$}} & {\small{AIDS induced death rate}} & {\small{$1$}} \\
\hline \hline
\end{tabular}
\end{table}
% ------------------------

\subsection{Local asymptotic stability of the disease free equilibrium $\Sigma_0$}

Consider the parameter values of Table~\ref{table:parameters} and $\beta =0.001$. 
The basic reproduction number \eqref{eq:R0:model:1} is
$$
R_0 = 0.79587
$$ 
while the disease free equilibrium \eqref{DFE} takes the value
$$
\Sigma_0 = \left(\frac{\Lambda}{\mu},0, 0,0  \right) = \left( 146.034, 0, 0, 0  \right). 
$$
On the other hand, the discriminant $D(p)$ of the polynomial $p$ \eqref{eq:poly:p} 
is given by $D(p) = 0.51045 > 0$, $b_1 = 2.41711 > 0$, 
$b_3 =  0.00652 > 0$ and $b_1 b_2 b_3 = 0.02205 > 0$. Therefore, 
the Routh--Hurwitz conditions are a necessary and sufficient condition 
for the equilibrium point $\Sigma_0$ to be locally
asymptotically stable (see Section~\ref{sec:local:DFE}). 
The stability of the disease free equilibrium $\Sigma_0$ is illustrated 
in Figure~\ref{fig:SICA:DFE:07to1}, where we considered the initial conditions 
\begin{equation*}
\label{init_cond}
S(0)=0.8,\quad I(0)=0.1,\quad C(0)=0,\quad A(0)=0
\end{equation*}
and a fixed time step size of $h=2^{-6}$.

For the numerical implementation of the fractional derivatives, 
we have used the Adams--Bashforth--Moulton scheme, which has been 
implemented in the \textsf{Matlab} code \textsf{fde12} 
by Garrappa \cite{Garrappa}. This code implements a predictor-corrector 
PECE method of Adams--Bashforth--Moulton type, as described in \cite{Diethelm:1999}.

Regarding convergence and accuracy of the numerical method, we refer to \cite{Diethelm:2004}.
The stability properties of the method implemented by \textsf{fde12} 
have been studied in \cite{Garrappa:2010}. Here we considered, without loss of generality, 
the fractional-order derivatives $\alpha = 1.0, 0.9, 0.8$ and $0.7$.
% ---------------------------------------------------------------------------
\begin{figure}[htb]
\centering
\subfloat[\footnotesize{$S(t)$ for $\alpha \in \{0.7, 0.8, 0.9 1\}$ 
and $t \in [0, 10000]$.}]{\label{S:alpha07to1}
\includegraphics[width=0.45\textwidth]{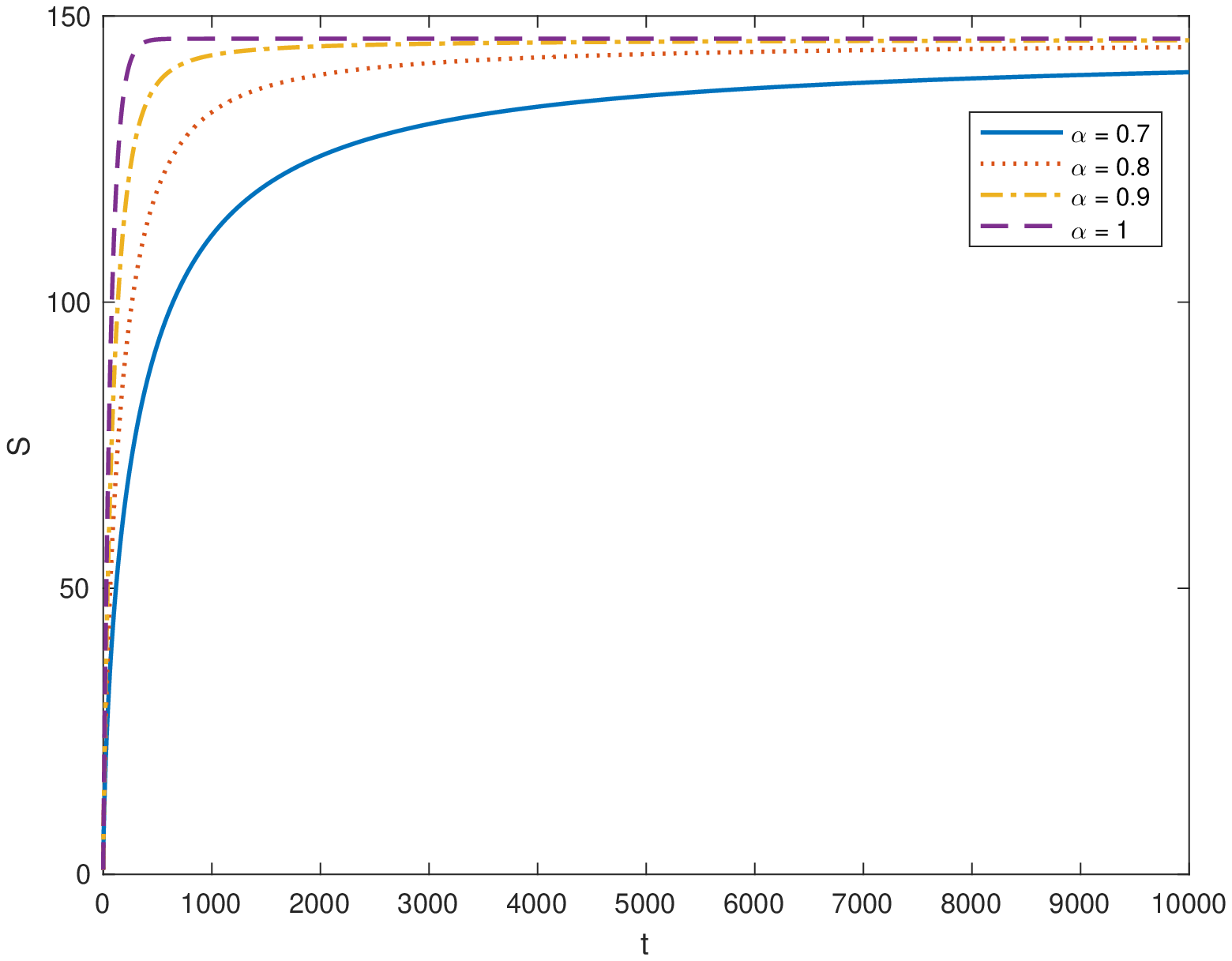}}
\subfloat[\footnotesize{$I(t)$ for $\alpha \in \{0.7, 0.8, 0.9 1\}$ 
and $t \in [0, 10000]$.}]{\label{I:alpha07to1}
\includegraphics[width=0.45\textwidth]{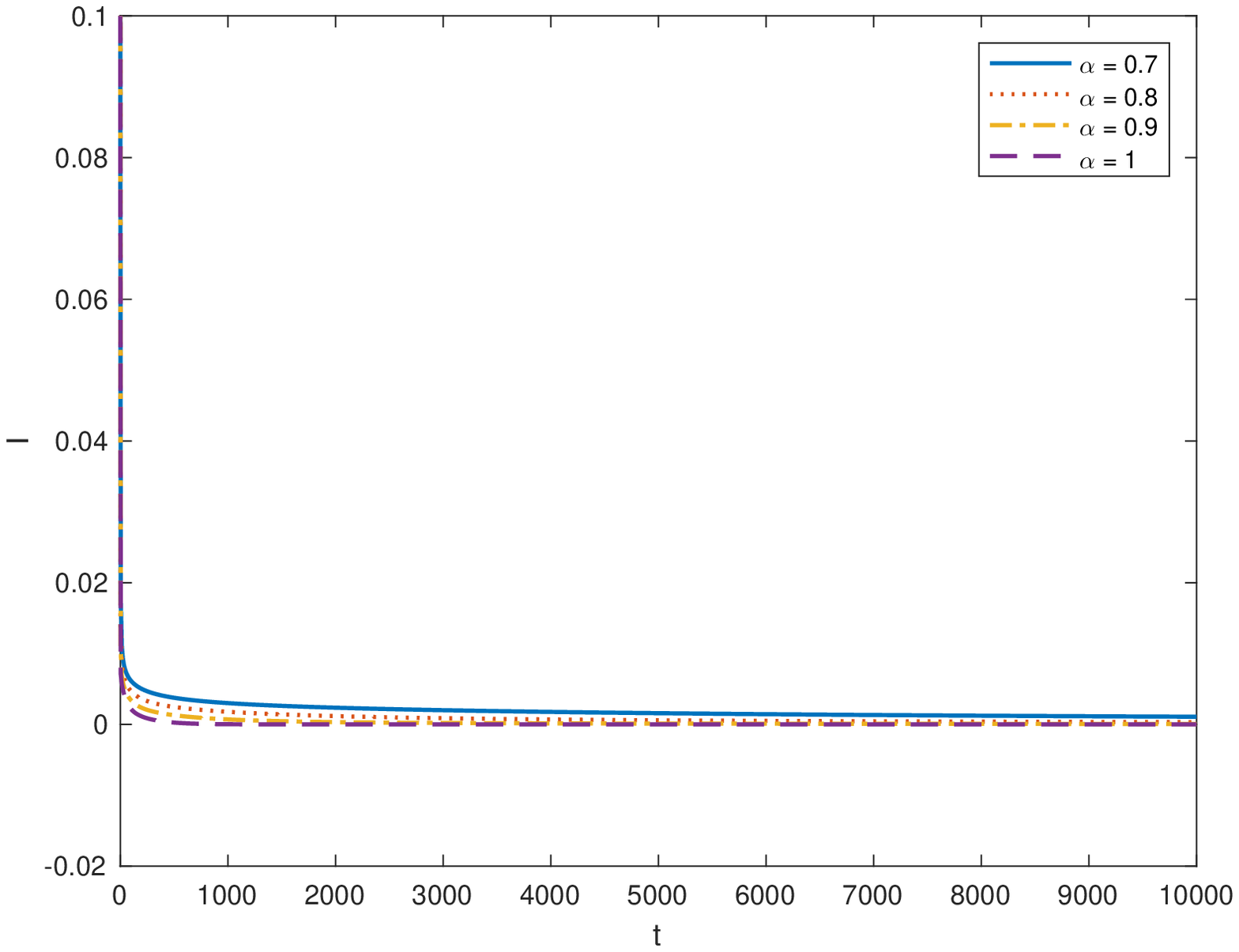}}\\
\subfloat[\footnotesize{$C(t)$ for $\alpha \in \{0.7, 0.8, 0.9 1\}$ 
and $t \in [0, 10000]$.}]{\label{C:alpha07to1}
\includegraphics[width=0.45\textwidth]{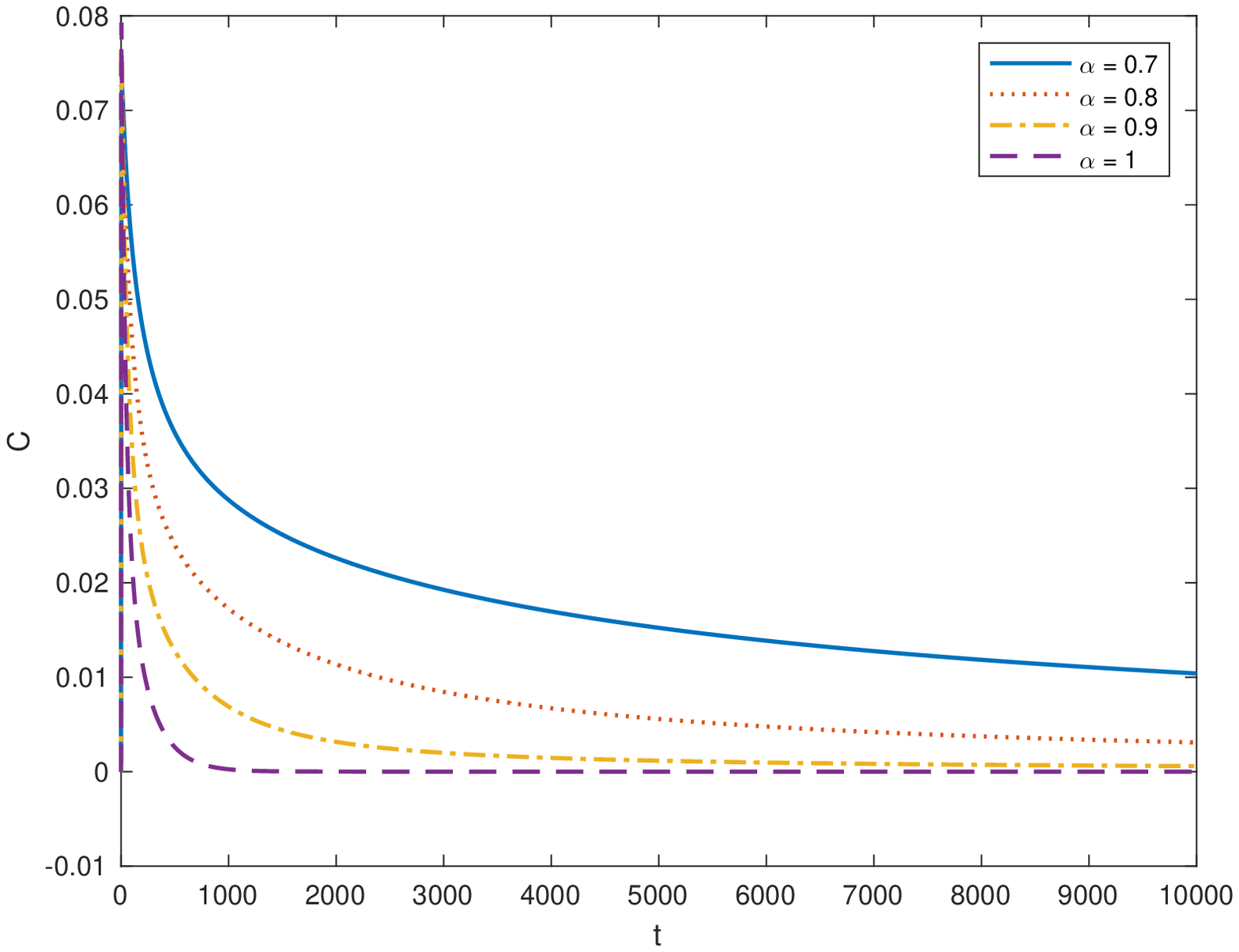}}
\subfloat[\footnotesize{$A(t)$ for $\alpha \in \{0.7, 0.8, 0.9 1\}$ 
and $t \in [0, 10000]$.}]{\label{A:alpha07to1}
\includegraphics[width=0.45\textwidth]{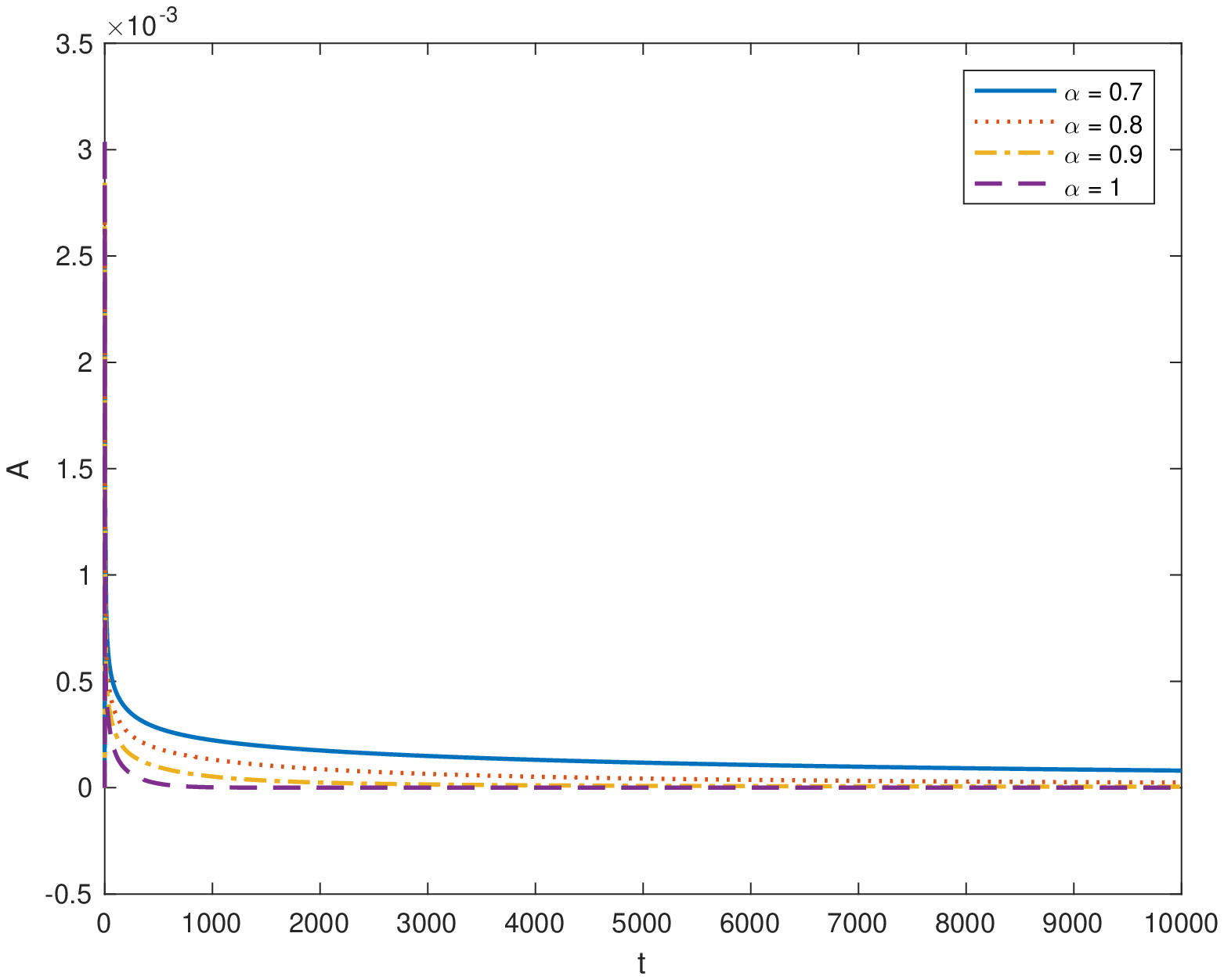}}
\caption{Stability of the disease free equilibrium $\Sigma_0$.}
\label{fig:SICA:DFE:07to1}
\end{figure}

% ---------------------------

\subsection{Stability of the endemic equilibrium $\Sigma_*$}

For the numerical study of the stability of the endemic equilibrium 
$\Sigma_*$ \eqref{EE}, we consider the parameter values from Table~\ref{table:parameters} 
and $\beta = 0.01$. The basic reproduction number \eqref{eq:R0:model:1} 
takes the value $R_0 = 7.95871$. The concrete value of the endemic equilibrium 
\eqref{EE} is $\Sigma_* = \left( 18.3490, 8.0673, 77.2881, 0.6001 \right)$. 
Figure~\ref{fig:SICA:EE:07to1} illustrates the stability of the endemic equilibrium
for the initial conditions 
\begin{equation*}
S(0)=100,\quad I(0)=1,\quad C(0)=0,\quad A(0)=0,
\end{equation*}
where a fixed time step size of $h=2^{-6}$ has been used.
% ---------------------------------------------------------------------------
\begin{figure}[htb]
\centering
\subfloat[\footnotesize{$S(t)$ for $\alpha \in \{0.7, 0.8, 0.9 1\}$ 
and $t \in [0, 10000]$.}]{\label{S:alpha07to1:EE}
\includegraphics[width=0.45\textwidth]{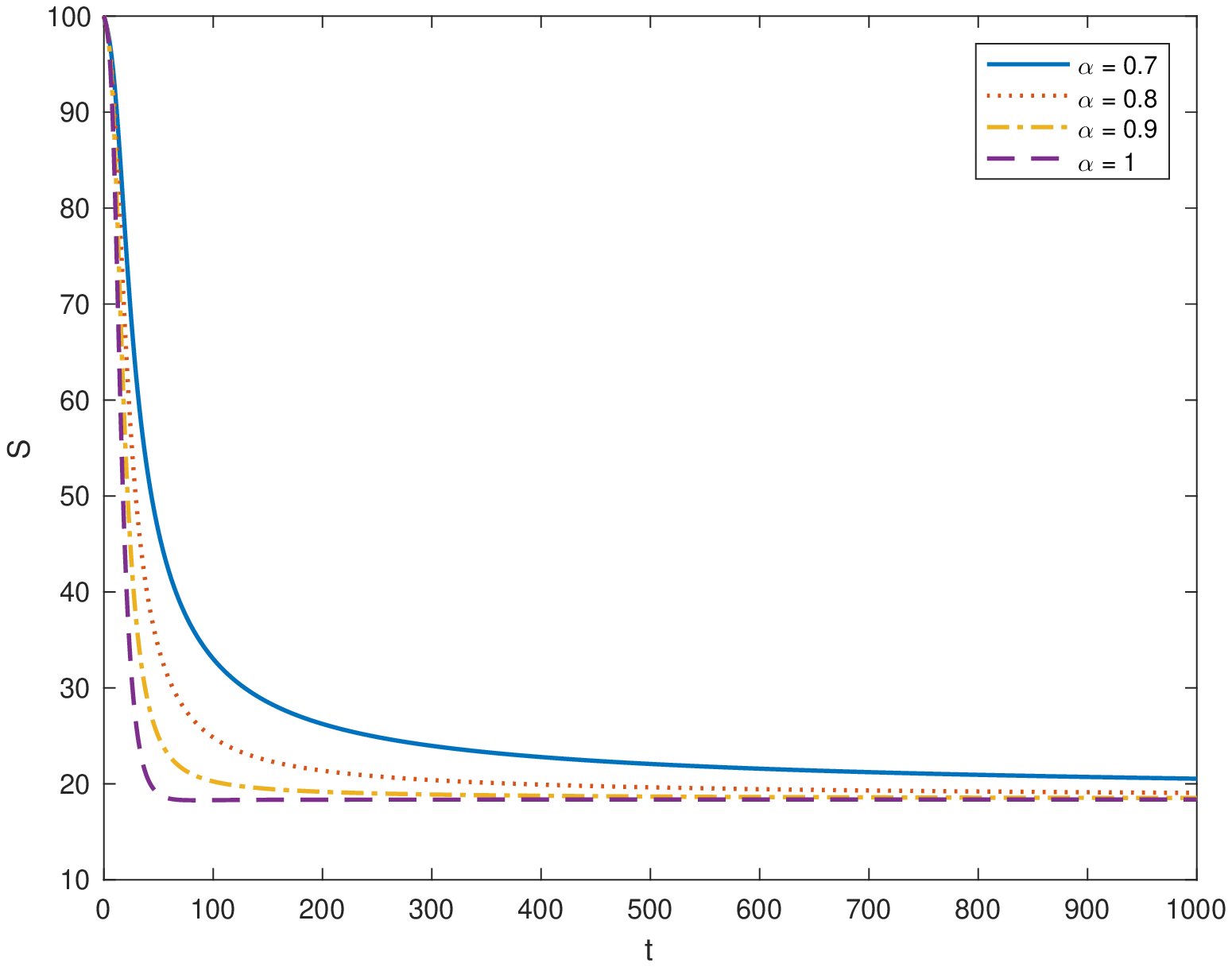}}
\subfloat[\footnotesize{$I(t)$ for $\alpha \in \{0.7, 0.8, 0.9 1\}$ 
and $t \in [0, 10000]$.}]{\label{I:alpha07to1:EE}
\includegraphics[width=0.45\textwidth]{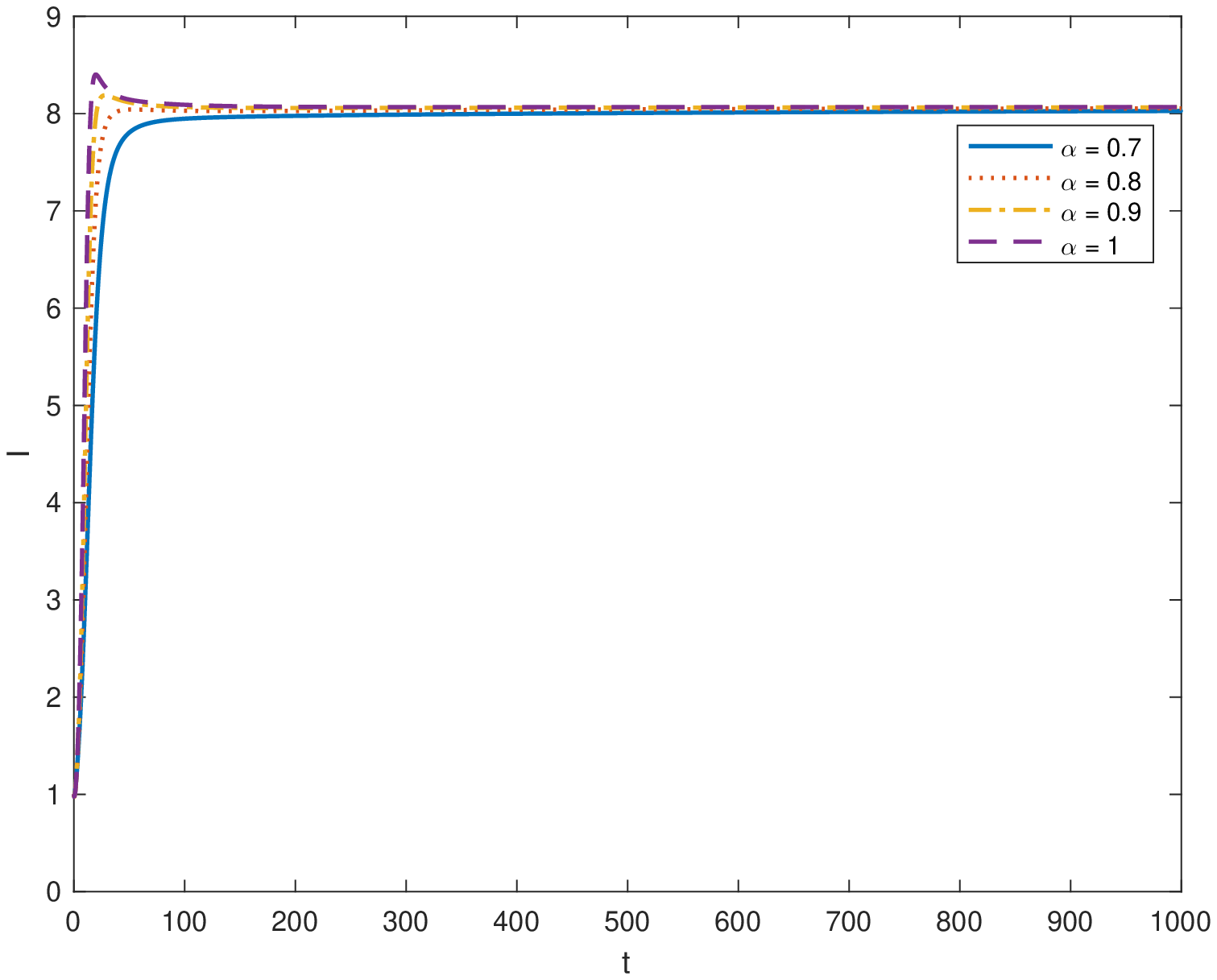}}\\
\subfloat[\footnotesize{$C(t)$ for $\alpha \in \{0.7, 0.8, 0.9 1\}$ 
and $t \in [0, 10000]$.}]{\label{C:alpha07to1:EE}
\includegraphics[width=0.45\textwidth]{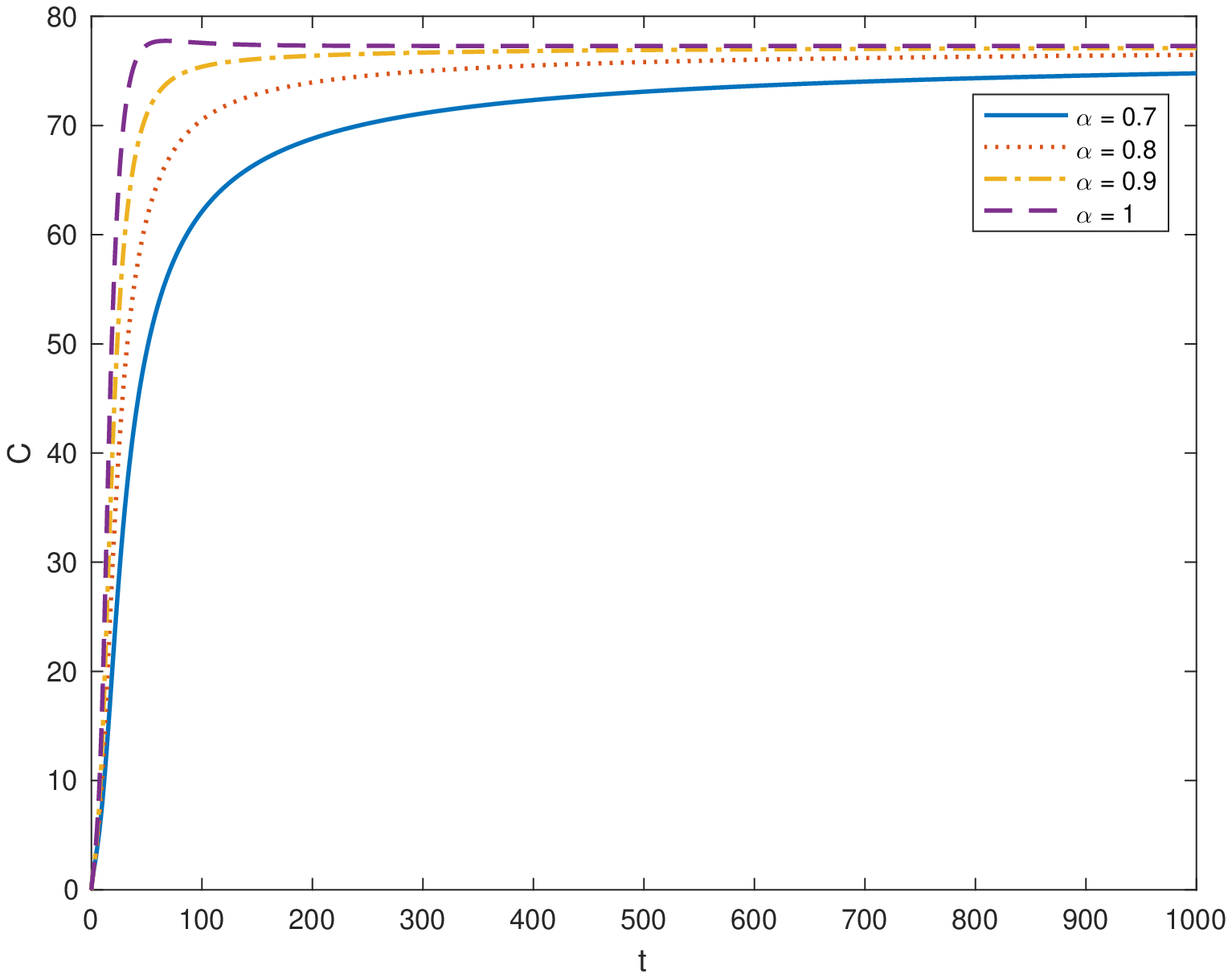}}
\subfloat[\footnotesize{$A(t)$ for $\alpha \in \{0.7, 0.8, 0.9 1\}$ 
and $t \in [0, 10000]$.}]{\label{A:alpha07to1:EE}
\includegraphics[width=0.45\textwidth]{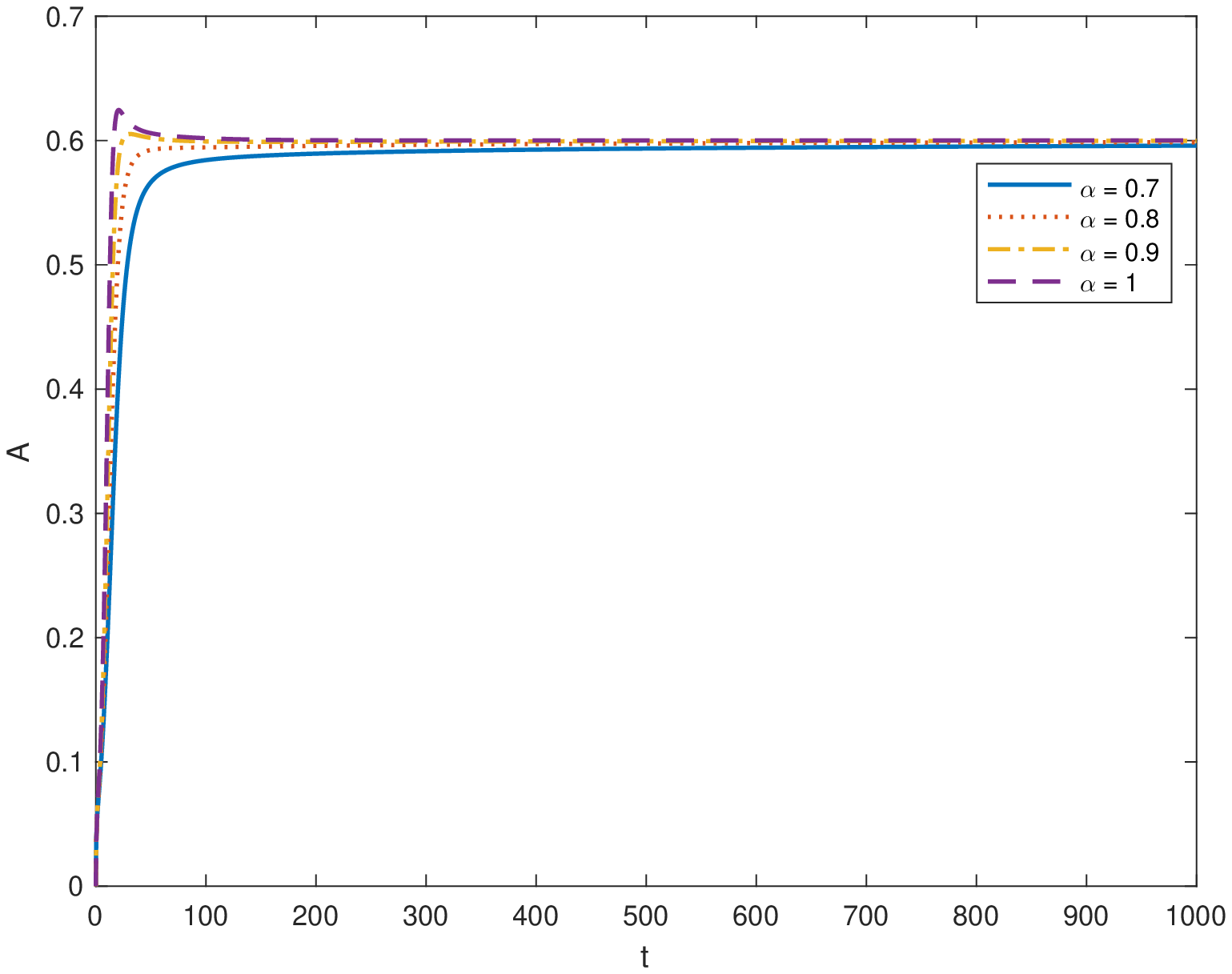}}
\caption{Stability of the endemic equilibrium $\Sigma_*$.}
\label{fig:SICA:EE:07to1}
\end{figure}

Our results show that the smaller the order $\alpha$ 
of the fractional derivative, the slower the convergence 
to the equilibrium point.

% ---------------------------------------------

\section*{Acknowledgments}

This research was partially supported by the Portuguese Foundation 
for Science and Technology (FCT) through the R\&D unit CIDMA, 
reference UID/MAT/04106/2019, and by project PTDC/EEI-AUT/2933/2014 (TOCCATA), 
funded by FEDER funds through COMPETE 2020 -- Programa Operacional Competitividade 
e Internacionaliza\c{c}\~ao (POCI) and by national funds through FCT.
Silva is also supported by national funds (OE), through FCT, I.P., in the scope
of the framework contract foreseen in the numbers 4, 5 and 6
of the article 23, of the Decree-Law 57/2016, of August 29,
changed by Law 57/2017, of July 19. The authors  are grateful
to three reviewers for their critical remarks and
precious suggestions, which helped them to improve 
the quality and clarity of the manuscript.

% -------------------

% -------------------

\end{document}